\documentclass[12pt, a4paper, leqno]{amsart}
\usepackage[utf8]{inputenc}
\usepackage{amsmath}
\usepackage{amsfonts}
\usepackage{amssymb}
\usepackage{times}
\usepackage[T1]{fontenc} 
\usepackage{url} 
\usepackage{color,esint}
\usepackage[colorlinks, linkcolor=blue, citecolor=blue]{hyperref} 
\usepackage{graphicx} 
\usepackage{enumitem}
\usepackage{tabularx}
    \usepackage{mathtools}
\usepackage{comment}

\let\oldmarginpar\marginpar
\renewcommand\marginpar[1]{\-\oldmarginpar[\raggedleft\footnotesize #1]%
{\raggedright\footnotesize #1}}

\usepackage{amsthm}
\theoremstyle{plain}
\newtheorem{thm}[equation]{Theorem}
\newtheorem{lem}[equation]{Lemma}

\newtheorem{cor}[equation]{Corollary}

\theoremstyle{definition}
\newtheorem{defn}[equation]{Definition}

\theoremstyle{remark}
\newtheorem{rem}[equation]{Remark}

\newtheorem{example}[equation]{Example}

\numberwithin{equation}{section}

\newcommand{\R}{\mathbb{R}}

\newcommand{\loc}{\mathrm{loc}}

\renewcommand{\div}{\divop}
\def\le{\leqslant}
\def\leq{\leqslant}
\def\ge{\geqslant}
\def\geq{\geqslant}
\def\phi{\varphi}
\def\rho{\varrho}
\def\vartheta{\theta}

\newcommand{\iprod}{\mathbin{\lrcorner}}

\def\esssup{\operatornamewithlimits{ess\,sup}}

\def\div{\qopname\relax o{div}}

\date{\today}

\begin{document}

\title[]{{Calder\'on-Zygmund estimates for nonlinear equations of differential forms with BMO coefficients}
}

\author{Mikyoung Lee}
  \address{Mikyoung Lee, Department of Mathematics and Institute of Mathematical Science, Pusan National University, Busan
46241, Republic of Korea}
\email{mikyounglee@pusan.ac.kr}

\author{Jihoon Ok}
\address{Jihoon Ok, Department of Mathematics, Sogang University, Seoul 04107, Republic of Korea}
\email{\texttt{jihoonok@sogang.ac.kr}}

\author{Juncheol Pyo}
  \address{Juncheol Pyo, Department of Mathematics and Institute of Mathematical Science, Pusan National University, Busan
46241, Republic of Korea}
\email{jcpyo@pusan.ac.kr}

\thanks{MSC(2020) 35J92, 35B65, 35D30}

\subjclass[2020]{} 


\begin{abstract} 
We obtain $L^q$-regularity estimates for weak solutions to $p$-Laplacian type equations of differential forms. In particular,
we prove local Calder\'on-Zygmund type estimates for equations with discontinuous coefficients satisfying the bounded mean oscillation (BMO) condition.
\end{abstract}

\maketitle


\section{Introduction}

In this paper we study regularity theory for the following nonhomogenous quasilinear system:
\begin{equation}\label{mainpb}
d^* \big(a(x)|du|^{p-2}du \big)= d^*(F) \quad \textrm{ in }\ \Omega ,
\end{equation}
where $1<p<\infty$, $\Omega$ is 
an open set in $\R^n$ with $n\ge 2$, $a: \Omega \to [\nu,L]$ with $0<\nu\le L$, $u$ is a nonsmooth $\ell$-form from $\Omega$ with an integer $\ell = 0, 1,\dots, n-1$, and  $F\in L^{\frac{p}{p-1}}(\Omega,\Lambda^{\ell+1})$. In this context,  $d$ is the weak exterior derivative and $d^*$ is the Hodge codifferential with respect to $d$ (see next section).

We note that when $\ell = 0$, the exterior derivative $d$ is the gradient $\nabla$ and the Hodge codifferential $d^*$ is the divergence $\mathrm{div}$. (See Example~\ref{example1}.)
In this case, the Calder\'on-Zygmund estimates in the Lebesgue spaces $L^q$ for equations or systems involving the $p$-Laplacian with coefficients are well-known, and extensive research has been conducted on both cases, establishing them as fundamental operators in the analysis of quasilinear elliptic equations and systems. Iwaniec \cite{Iwa83} first obtained Calder\'on-Zygmund type estimate for the $p$-Laplace equation with $p>2$. Later DiBenedetto  and Manfredi \cite{DM93} generalized Iwaniec's result to the system case with $1<p<\infty$, and Kinnunen and Zhou \cite{KZ99} have successfully extended  to anisotropic elliptic equations with  the vanishing mean oscillation (VMO) coefficient. We note that this result is naturally extended to equations with BMO coefficients with small BMO seminorm depending on the exponent $q$.
Recently, in \cite{BDGP22}, Balci, Diening, Giova and Passarelli di Napoli considered degenerated coefficients and obtained a sharp upper bound of the BMO seminorm with respect to the exponent $q$.
 We refer to \cite{AM05,AM07,BW12,CM16} for further results of Calder\'on-Zymund estimates for elliptic and parabolic problems, as well as those with nonstandard growth.

In the realm of the case $\ell \ge 0$ , Uhlenbeck \cite{Uhl77} made significant contributions by exploring the broad scope of the following homogeneous elliptic systems 
\begin{equation}\label{Uhlenbeck}
d^* (g(|\omega|)\omega) =0 , \quad d\omega=0,
\end{equation}
where $\omega$ is an $(\ell+1)$-form, $g(t)\approx |t|^{p-2}$ in some sense with $p \ge 2$. In particular, she proved the H\"older continuity of $\omega$.  As an extension of her result in  the framework of differential forms
on a Riemannian manifold with sufficiently smooth boundary, Hamburger \cite{Ham92}  established global H\"older regularity for more general class of systems with the Neumann and Dirichlet boundary condition. Beck and Stroffolini \cite{BS13} studied partial regularity for anisotropic systems with $p$-growth. We also  refer to \cite{GH04,Ham05} for further research on regularity theory for systems or functionals  related to the system \eqref{Uhlenbeck}.

On the contrary, in \cite{BDS15}, Bandyopadhyay, Dacorogna and Sil explored functionals of the form
$$
\int_\Omega f(du)\,dx
$$
and developed into the existence theory of minimizers corresponding to variational problems, see also \cite{BS16}.   We note that if $f(\xi)=|\xi|^p$ the corresponding Euler-Lagrange equation is
$$
d^* \left(|du|^{p-2}du\right) =0 \quad \text{in }\ \Omega,
$$
which corresponds to \eqref{Uhlenbeck} in the exact case, i.e., $\omega= d u$ for some $\ell$-form $u$, together with $g(t)=t^{p-2}$.  In fact, if the simply-connected domain $\Omega$ and the closed $(\ell+1)$-form $\omega$ are smooth enough, thanks to the well-known Poincar\'e lemma, we can always find an $\ell$-form $u$ such that $\omega=du$. Furthermore, in subsequent research, Sil obtained some regularity results for related problems. In \cite{Sil17}, he studied Calder\'on-Zygmund type estimates for  linear systems of differential forms with $C^{1,\alpha}$ coefficients and obtained $W^{2,p}$-estimates. Moreover, in \cite{Sil19}, he considered the following nonlinear systems of $p$-Laplacian type 
\begin{equation}\label{mainpb1}
d^* \big(a(x)|du|^{p-2}du \big)= f \quad \textrm{ in }\ \Omega,
\end{equation}
and proved that $du$ is continuous if $f\in L^{n,1}_{\loc}(\Omega)$ with $d^*f=0$ and  $a:\Omega\to [\nu,L]$ is Dini continuous.
 This is inspired by the result of Kuusi and Mingione \cite{KMin14} for $p$-Laplace systems.  Note that in this paper, we are interested in regularity theory when the coefficient function $a$ in \eqref{mainpb1} is discontinuous and $f\in L^\gamma$ with $\gamma<n$.   We also refer to \cite{CDS18,FG16,KS23,Sil19-2} for further research on regularity theory for equations with differential forms.

 \subsection{Main result}

 We write  $B_\rho(y)$ for  the open ball in $\R^n$ with center $y\in \R^n$ and radius $\rho>0$. 
 For the sake of simplicity, we write $B_\rho=B_\rho(y)$ if the center is clear from the context.
 For a measurable function $g:U\to \mathbb{R}$ with $U\subset \mathbb{R}^n$, we define  the average of $g$ in $U$ by
 \begin{equation}\label{Ave}
(g)_U:=\fint_{U} g \; dx = \frac{1}{|U|} \int_{U} g \;dx. 
\end{equation}  If we have $\overline{\Omega'} \subset \Omega$  for  a bounded and open set $\Omega'$, we denote it simply as $\Omega' \Subset \Omega$.

The aim of this paper is  to
 derive Calderón-Zygmund type $L^q$-estimates for a local weak solution to the equation \eqref{mainpb}.
Let us introduce the definition of weak solution we are considering in this paper.
 \begin{defn}
  We say that  $u\in W_{\loc}^{1,p}(\Omega,\Lambda^\ell)$ is a local \textit{weak solution} to \eqref{mainpb} if it satisfies the weak formulation 
\begin{equation}\label{weakform}
\int_{B_R} a(x)|du|^{p-2} \langle du, d\phi \rangle \,dx = \int_{B_R} \langle F, d \phi \rangle \,dx 
\end{equation}
for all $\phi\in W_0^{1,p}(B_R,\Lambda^\ell)$ with $B_R\Subset \Omega$.
\end{defn}

\begin{rem} \label{rmk:weakform}
In view of \cite[Proposition 3]{Sil19}, if $u\in W_{\loc}^{1,p}(\Omega,\Lambda^\ell)$ is a weak solution,  then \eqref{weakform} holds
for every function $\phi\in W^{1,p}_{d^*,T}(B_R,\Lambda^\ell) $ with $B_R\Subset \Omega$.
\end{rem}

Furthermore, we impose the condition that the coefficient function $a$ satisfies  the following BMO type condition:
\begin{defn}
We say that $a:\Omega\to \R$ is \textit{$(\delta,R)$-vanishing} with $\delta,R>0$ if 
$$
\sup_{r\in(0,R]}\sup_{B_r\subset \Omega} \fint_{B_r}|a(x)-(a)_{B_r}|\,dx \le \delta.
$$
\end{defn}
This condition means that the BMO seminorm $[a]_{BMO(B_R)} \le \delta$ for every $B_R\subset \Omega$, hence the coefficient function $a$ can be discontinuous.

We now state the main theorem of the paper. 
\begin{thm}\label{mainthm}
Let $u\in W^{1,p}_{\loc}(\Omega,\Lambda^\ell)$ be a local weak solution to \eqref{mainpb}. For every $q>p$, there exists a small $\delta=\delta(n,\ell,p,\nu,L,q)>0$ such that if the coefficient function $a$ is $(\delta,R)$-vanishing for some $R>0$ and $F\in L^{\frac{q}{p-1}}_{\loc}(\Omega,\Lambda^{\ell+1})$, then we have $du\in L^{q}_{\loc}(\Omega,\Lambda^{\ell+1})$ with the estimate: 
for every $B_{2r} \Subset \Omega$ with $2r \le R$,
$$
\fint_{B_r} |du|^{q}\,dx \le c \left(\fint_{B_{2r}} |du|^{p}\,dx \right)^{\frac{q}{p}} + c \fint_{B_{2r}} |F|^{\frac{q}{p-1}}\, dx 
$$
for some $c=c(n,\ell,p,\nu,L,q)>0$. 
\end{thm}

\begin{rem}
In the above theorem, if the local weak solution $u$ is coclosed, i.e., $d^* u=0$, then by Lemma~\ref{lem:Gaffneylocal}, we have $\nabla u\in L^{q}_{\loc}(\Omega,\Lambda^{\ell+1})$
with the estimate 
$$
\fint_{B_r} |\nabla u|^{q}\,dx \le c \left(\fint_{B_{2r}} |\nabla u|^{p}\,dx \right)^\frac{q}{p} + c \fint_{B_{2r}} |F|^{\frac{q}{p-1}}\, dx.
$$
\end{rem}

 \begin{rem}
We can obtain the same result as in the above theorem for the system case with $u=(u^1,u^2,\dots,u^N)\in W^{1,p}_{\loc}(\Omega;\Lambda^{\ell}\otimes\R^N)$ and $F=(F^1,F^2,\dots,F^N)\in W^{1,p}_{\loc}(\Omega;\Lambda^{\ell+1}\otimes\R^N)$.  Its proof is analogous to the one of the theorem.
 \end{rem}

From the above theorem, we can also obtain Calder\'on-Zygmund type estimates for the equations with non-codifferential data $f$ as in \eqref{mainpb1}. 

\begin{cor}\label{cor1}
Let $1<p<\infty$ and
$$
\gamma_0 :=
\left\{\begin{array}{cl}
\displaystyle \frac{np}{np-n+p} & \text{if }\ 1<p<n,\\
1& \text{if } \ p \ge n.
\end{array}\right.
$$
For every $\gamma \in (\gamma_0,n)$, there exists a small $\delta=\delta(n,\ell,p,\nu,L,\gamma)>0$ such that if the coefficient function $a$ is $(\delta,R)$-vanishing for some $R>0$, $f\in L^\gamma_{\loc}(\Omega,\Lambda^\ell)$ with  $d^* f=0$ in the distribution sense, and $u\in W^{1,p}_{\loc}(\Omega)$ is a weak solution to \eqref{mainpb1}, then we have  $du\in L^{\frac{n\gamma (p-1)}{n-\gamma}}_{\loc}(\Omega,\Lambda^{\ell+1})$ with the estimate: for every $B_{2r}\Subset\Omega$ with $2r \le R$,
$$
\fint_{B_r} |d u|^{\frac{n\gamma (p-1)}{n-\gamma}}\,dx \le c \left(\fint_{B_{2r}} |d u|^{p}\,dx \right)^{\frac{n\gamma (p-1)}{(n-\gamma)p}} + c \fint_{B_{2r}} |rf|^{\frac{n\gamma}{n-\gamma}}\, dx 
$$
for some $c=c(n,\ell,p,\nu,L,\gamma)>0$.
\end{cor}

We note in the above corollary that the conditions $f\in L^\gamma(\Omega)$ with $\gamma> \gamma_0$ and $d^* f=0$ in the distribution sense are natural from the existence of the weak solution, see \cite[Section 2.5]{Sil19} for more details.


 
 In the proof of the main theorem, we employ a perturbation technique widely used in the field of partial differential equations.  
 In particular, in the main part of the proof in Section~\ref{sec:proof}, we adopt the approach initially introduced by Acerbi and Mingione in \cite{AM07}, 
 where they proved Calder\'on-Zygmund type estimates for parabolic $p$-Laplacian type systems.
 We note that there have been other approaches to obtain Calder\'on-Zygmund estimates for nonlinear problems, employing maximal operators and covering arguments based on dyadic decomposition, see \cite{CP98,DM93,Iwa83}. In contrast, the method used in  \cite{AM07} takes advantage of more elementary tools such as Vitali covering lemma and density of supper-level sets without the need for maximal operators. Therefore it can be applied to a wide range of generalized problems.
 In detail, we employ exit time argument to a nonlinear functional with respect to $du$ and $F$ and use the Vitali covering lemma to create an appropriate collection of balls that covers the supper-level set for $|du|^p$. Subsequently, within each ball, we make comparisons between the exterior derivatives of the local weak solutions for both our main equation \eqref{mainpb} and its corresponding limiting equation which is homogeneous and has constant coefficient. 
 A critical facet of this stage involves achieving higher integrability for the exterior derivatives of the local weak solutions to our problem. This is accomplished in  Lemma~\ref{lem:high}  through the suitable application of the Sobolev-Poincar\'e inequality for differential forms. Within this process, the careful selection of a test function from the appropriate function spaces within the weak formulation assumes a vital role.  Then we take advantage of the known regularity result of the homogeneous system stated in Lemma~\ref{lem:locbdd_h}.

 \subsection{Examples}
 We shall introduce several cases of \eqref{mainpb}. Note that $u$ is the $\ell$-form and $F$ is the $(\ell+1)$-form.  In the following examples, we shall identify $1$-forms with vector fields in $\R^n$ by taking their dual. Using the Hodge star operator we identify $(n-1)$-forms with vector fields in $\R^n$ and $n$-forms with scalar functions.

 \begin{example}\label{example1}  ($p$-Laplacian type equation)
Suppose $\ell=0$.  Then $d u= \nabla u$ and $d^* \omega = \mathrm{div} \omega$ for every $1$-form $\omega$. Therefore \eqref{mainpb} becomes
$$
\div (a(x) |\nabla u|^{p-2} \nabla u) = \div (F) 
\quad \text{in }\ \Omega.
$$ 
 \end{example}

 \begin{example}  (Divergence equation)
Suppose $\ell=n-1$.  Then $d u= \div u$ and  $d^* \omega = \nabla \omega$ for every $n$-form $\omega$. Therefore \eqref{mainpb} becomes
$$
\nabla (a(x) |\div u|^{p-2}\div u) = \nabla F
\quad \text{in }\ \Omega,
$$
which implies that 
$$
a(x) |\div u|^{p-2}\div u  = F +c
\quad \text{in }\ \Omega
$$
for some constant $c$.
 \end{example}

 \begin{example}  (Maxwell type equation)
Suppose $n=3$ and $\ell=1$.  Note that in this case, $1$-forms and $2$-forms can be identified with vector fields in $\R^3$. Then $d u= \nabla \times u$ and  $d^* \omega = \nabla \times \omega$ for every $2$-form $\omega$. Therefore \eqref{mainpb} becomes
$$
\nabla \times (a(x) |\nabla \times u|^{p-2}\nabla \times u) = \nabla \times F
\quad \text{in }\ \Omega.
$$
 \end{example}

\subsection*{Oragnization} 
 The remainder of this paper is organized as follows.
  The subsequent section is dedicated to gathering fundamental definitions related to the theory of differential forms. Additionally, we recall crucial auxiliary lemmas concerning the Sobolev-Poincar\'e inequality for differential forms and the local boundedness of the exterior derivative of weak solutions to the homogeneous problem with constant coefficients. In Section~\ref{Sec_last}, we provide the comparison estimates together with the higher integrability of the exterior derivative of weak solutions to our problem. Our main theorem is proved in the last section.


\section{Preliminaries}
\label{sec:preliminaries}

\subsection{Notation and Sobolev spaces for differential forms}

Let $0\le \ell, k \le n$ be integers and let $\Omega\subset \mathbb{R}^n $ be open, bounded and smooth with $n\ge 2$.
In what follows, we denote by $\Lambda^\ell(\mathbb{R}^n)$ the set of all exterior $\ell$-forms over $\mathbb{R}^n$, or for simplicity by $\Lambda^\ell $.   We denote by $C^{\infty}(\Omega, \Lambda^\ell)$ and $C^{\infty}_0(\Omega, \Lambda^\ell)$ the space of smooth $\ell$-forms and the space of smooth $\ell$-forms with compact support in $\Omega$, respectively. Note that $C^\infty(\R^n,\Lambda^0) = C^{\infty}(\mathbb{R}^n)$.

We start by reviewing fundamental notation for smooth $\ell$-forms.
The exterior product of $\omega \in \Lambda^\ell$ and $\phi \in \Lambda^k$, denoted by $\omega \wedge \phi$, is the $(\ell+k)$-form defined by
\[
(\omega \wedge \phi) (X_1, \dots, X_{\ell+k}) = \sum \textrm{sign} (i_1,\dots,i_k,j_1,\dots,j_\ell) \omega(X_{i_1},\dots,X_{i_\ell})\phi(X_{j_1},\dots,X_{j_k})
\]
for every $(\ell+k)$-tuple vector fields $(X_1, \dots, X_{\ell+k})$ in $\mathbb{R}^n$, where the summation is over permutations $(i_1,\dots,i_k,j_1,\dots,j_\ell)$ of $(1,2,\dots, k+\ell)$ with $i_1<\cdots<i_\ell$ and  $j_1<\cdots<j_k.$

Let $\{e_i\}_{i=1}^n$ be orthonormal vector fields in $\mathbb{R}^n$ and  $\{e^i\}_{i=1}^n$ be their dual 1-forms in the sense $e^i(e_j) =\delta_{ij}$. Note that $\{e^{i_1}\wedge\cdots \wedge e^{i_\ell} : 1\le i_1 < \cdots < i_\ell \le n \}$ is a basis of $\Lambda^\ell$ whose dimension is ${n}\choose{\ell}$. 
Then any $\omega \in \Lambda^\ell$ can be represented as 
\[
\omega = \sum_{1\le i_1 < \cdots <i_\ell\le n} \omega_{i_1 \cdots i_\ell} e^{i_1} \wedge\cdots \wedge e^{i_\ell}
\]
for uniquely determined $\omega_{i_1 \dots i_\ell}  \in C^{\infty}(\mathbb{R}^n)$ with respect to $\{e_i\}_{i=1}^n$. Moreover,  the coefficients of $\omega$ can be recovered by the formula $\omega_{i_1 \dots i_\ell} = \omega(e_{i_1}, \dots, e_{i_\ell})$.
 The scalar product of $\omega, \phi \in \Lambda^\ell$ with coefficients $\omega_{i_1 \cdots i_\ell}$ and  $\phi_{i_1 \cdots i_\ell}$ is defined as 
\[
\langle \omega, \phi \rangle =\sum_{1\le i_1 < \cdots <i_\ell\le n} \omega_{i_1 \cdots i_\ell} \phi_{i_1 \cdots i_\ell} .
\]  The scalar product is independent of the chosen orthonormal vector fields $\{e_i\}_{i=1}^n$.
We denote 
$
|\omega|^2:= \langle \omega, \omega \rangle.$

The Hodge star operator is the linear map $* : \Lambda^\ell \rightarrow \Lambda^{n-\ell}$ defined by 
\[
\omega \wedge \phi = \langle * \omega, \phi \rangle e^1\wedge \cdots \wedge e^n,
\]
for any $\omega\in \Lambda^\ell$, $\phi \in \Lambda^{n-\ell}$. Then we note that  $ *(e^1 \wedge \cdots \wedge e^n) =1$ and   $*1 = e^1 \wedge \cdots \wedge e^n,$
\[
\omega \wedge (*\phi) = \langle \omega, \phi \rangle e^1\wedge \cdots \wedge e^n, \text{ and } *(*\omega) = (-1)^{\ell(n-\ell)} \omega
\]
for every $\omega, \phi \in \Lambda^\ell$.
The interior product (contraction) of $\omega \in \Lambda^\ell$ with $\phi \in \Lambda^k$, denoted by $\phi \iprod \omega$, is the $(\ell-k)$-form defined by
\[
\phi\iprod \omega = (-1)^{n(\ell-k)} * (\phi \wedge (*\omega)).
\]
 Note that for every $\omega \in \Lambda^\ell, \phi \in \Lambda^{\ell+1}$ and $\theta \in  \Lambda^1$, 
\[
|\theta|^2 \omega = \theta\iprod(\theta\wedge \omega) + \theta\wedge (\theta \iprod \omega)  \text{ and } \ \langle \theta \wedge \omega, \phi \rangle =  \langle \omega,  \theta \iprod   \phi \rangle.
\]

As usual, we denote by $d$ the exterior derivative 
$$ d: C^{\infty}(\Omega, \Lambda^\ell) \rightarrow C^{\infty}(\Omega, \Lambda^{\ell+1})
$$
for which we have a product formula of the form
$$ d(\omega_1 \wedge \omega_2) = d \omega_1 \wedge \omega_2 + (-1)^\ell \omega_1 \wedge d\omega_2,
$$
where $\ell$ stands for the degree of $\omega_1$.

Usually, differential forms are understood as smooth differential
forms, but we will deal with nonsmooth differential forms in this paper, i.e.
a differential $\ell$-form $\omega$ is a measurable function $\omega : \Omega \rightarrow \Lambda^\ell$.
 The exterior product, scalar product, the Hodge star operator, and interior product are naturally defined on nonsmooth differential forms.
Let $\{ dx^i\}_{i=1}^{n}$ be the dual basis of the standard orthonormal basis for $\mathbb{R}^n$ and the Euclidean volume form is denoted by $dx := dx^1 \wedge \cdots \wedge dx^n$.

We denote by $L^p(\Omega, \Lambda^\ell), 1\le p \le \infty$, the Lebesgue space of all (measurable) differential $\ell$-forms $\omega$ on $\Omega$ for which 
\[
\|\omega \|_{L^p(\Omega,\Lambda^\ell)} :=
\begin{cases}
 \bigg(\int_{\Omega} |\omega|^p\,dx\bigg)^{\frac1p}  & \quad \text{if}\ 1\le p <\infty,\\
\esssup_{\Omega} |\omega| & \quad \text{if }\ p=\infty\\
\end{cases}
\]
is finite with norm $|\omega|=\langle \omega,\omega \rangle^{\frac12}$ inherited form the scalar product. For two differential forms $\omega \in L^p(\Omega, \Lambda^\ell)$ and $\varphi \in L^q(\Omega, \Lambda^\ell)$ with $1\le p, q \le \infty, \
\frac1p + \frac1q =1$, their scalar product in the sense of $L^2(\Omega, \Lambda^\ell)$ is finite and given by 
\[
(\omega, \varphi)_{\Omega} : = \int_{\Omega} \omega \wedge * \varphi = \int_{\Omega}   * \varphi \wedge \omega = \int_{\Omega} \langle \omega, \varphi \rangle \, dx.
\]
Analogous to the concept of weak derivatives, we introduce the definition of weak exterior derivatives.
\begin{defn}
For a differential $\ell$-form $\omega \in L_{loc}^{1}(\Omega, \Lambda^\ell)$, we say that
 a differential  $(\ell+1)$-form $\varphi \in L_{loc}^{1}(\Omega, \Lambda^{\ell+1})$ is called the (weak) exterior derivative of $\omega$, denoted by $d\omega$, if  
\[
\int_{\Omega} \eta \wedge  \varphi = (-1)^{n-\ell} \int_{\Omega} d\eta \wedge \omega
\] 
for any $\eta \in C^{\infty}_0(\Omega, \Lambda^{n-\ell-1})$.
\end{defn}
Note that the exterior derivative is unique.
The Hodge codifferential of $\omega \in L_{loc}^{1}(\Omega, \Lambda^\ell) $ is the differential $(\ell-1)$-form 
 defined by 
\[
d^*\omega :=(-1)^{n\ell+1} * d *\omega  \in L_{loc}^{1}(\Omega, \Lambda^{\ell-1}) .
\]
For a differential $\ell$-form $\omega$, we say that $\omega$ is closed if $d\omega =0$, and $\omega$ is coclosed if $d^* \omega =0$. We also say that $\omega$ is exact if $\omega = d\varphi$ for some differential $(\ell-1)$-form $\varphi$, and $\omega$ is coexact if $\omega = d^* \varphi$ for some differential $(\ell+1)$-form $\varphi$. Note that every exact form is closed and every coexact form is coclosed; that is $dd \omega =0$ and $d^*d^* \omega=0$.
  See \cite{CDK12} for the more properties and the integration by parts formula regarding these operators.

We use the standard definitions of the Sobolev spaces $W^{1,p}(\Omega, \Lambda^\ell)$ and the H\"older spaces $C^{0,\alpha}(\Omega, \Lambda^\ell)$ for $\alpha \in [0,1]$
  with the natural norms by requiring that each component belongs to the scalar versions of the corresponding spaces.
We can also consider the space $W^{1,p}_0(\Omega, \Lambda^\ell)$ given as the  closure of $C^{\infty}_0(\Omega, \Lambda^\ell)$
in $W^{1,p}(\Omega, \Lambda^\ell)$.
We use some additional Sobolev type spaces specifically suitable for forms.
We denote by 
\[
W^{d,p}(\Omega, \Lambda^\ell): = \{ \omega \in L^{p}(\Omega, \Lambda^\ell) : d\omega \in L^{p}(\Omega, \Lambda^{\ell+1})\},
\]
and
\[
W^{d^*,p} (\Omega, \Lambda^\ell): = \{ \omega \in L^{p}(\Omega, \Lambda^\ell) : d^*\omega \in L^{p}(\Omega, \Lambda^{\ell-1})\}
\]
the so-called partly Sobolev spaces of differential forms via the operators $d$ and $d^*$ respectively. Both spaces are Banach spaces when equipped with the norms
\[
\| \omega \|_{W^{d, p}} := \| \omega\|_{L^p} + \| d\omega\|_{L^p} \text{ and }  \| \omega \|_{W^{d^*, p}} := \| \omega\|_{L^p} + \| d^*\omega\|_{L^p},
\]
respectively. 
Note that the Sobolev space $W^{1,p}(\Omega, \Lambda^{\ell})$ coincides with the intersection of the two partly Sobolev spaces $W^{d,p}(\Omega, \Lambda^{\ell})$ and $W^{d^*,p}(\Omega, \Lambda^{\ell})$.

 Let $\nu$ be the outward unit normal to $\partial \Omega$ which is denoted $\nu= \sum_{i=1}^n \nu_i e_i$ with the standard basis in $\R^n$, it is identified with the $1$-form $\nu = \sum_{i=1}^n \nu_i dx^i$. For any differential $\ell$-form $\omega$ on $\Omega$, the differential $(\ell+1)$-form $\omega_T := \nu \wedge \omega$ is called the tangential part of $\omega$ on $\partial \Omega$ and the differential $(\ell-1)$-form $\omega_N := \nu \iprod  \omega$ is called the normal part of $\omega$ on $\partial \Omega$, interpreted in the sense
of traces when $\omega$ is discontinuous. 
 Furthermore, the spaces $W^{1,p}_T(\Omega, \Lambda^\ell)$, $W^{1,p}_N(\Omega, \Lambda^\ell)$,   and $W^{1,p}_{d^*, T} (\Omega, \Lambda^\ell)$ are defined as
 \[
W^{1,p}_T (\Omega, \Lambda^\ell): = \{ \omega \in W^{1, p} (\Omega, \Lambda^\ell): w_T = 0 \text{ in } \partial\Omega \},
\]
 \[
W^{1,p}_N (\Omega, \Lambda^\ell): = \{ \omega \in W^{1, p} (\Omega, \Lambda^\ell): w_N = 0 \text{ in } \partial\Omega \},
\]
and
\[
W^{1,p}_{d^*, T} (\Omega, \Lambda^\ell): = \{ \omega \in W^{1, p}_T (\Omega, \Lambda^\ell): d^* \omega = 0 \text{ in } \Omega \}.
\]
 We refer to \cite{CDK12, IL93, ISS99} for a more extensive discussion on Sobolev type spaces of differential forms. 

\begin{lem}\label{lem:Gaffneyglobal}
For every $p \in (1,\infty)$, we have 
\[
\| \omega \|_{W^{1,p}(\Omega)} \le c (\| \omega \|_{L^{p}(\Omega)} +  \| d\omega \|_{L^{p}(\Omega)} + \| d^* \omega \|_{L^{p}(\Omega)} )
\]
for all $\omega \in W_T^{1,p}(\Omega, \Lambda^{\ell}) \cup W_N^{1,p}(\Omega, \Lambda^{\ell})$ and for some constant $c = c(p, \Omega)>0$.
\end{lem}

The inequality in the above lemma is often called Gaffney inequality,
which is closely connected to Sobolev spaces and partly Sobolev spaces (see \cite[Theorem~4.8]{ISS99}).
Moreover, we have the following local version of the Gaffney inequality, for which we refer to \cite[Theorem 6]{Sil19}.
\begin{lem}\label{lem:Gaffneylocal}
For every $1<p<q<\infty$, we have 
\[
\fint_{B_r} |\nabla \omega|^{q} \,dx\le c \left(\fint_{B_{2r}} |\nabla \omega|^{p} \,dx\right)^q + c \fint_{B_{2r}} |d \omega|^{q} \,dx +c \fint_{B_{2r}} |d^* \omega|^{q} \,dx
\]
for all $\omega \in W^{1,p}(B_{2r}, \Lambda^{\ell})$ 
 and for some constant $c = c(p, q)>0$.
\end{lem}

From the next lemma, we can always assume that a local weak solution $u\in W^{1,p}_{\loc}(\Omega,\Lambda^\ell) $ to \eqref{mainpb} is coclosed in any $B_r \Subset \Omega$, i.e., $d^*u=0$ in $B_r$. The proof of the lemma is analogous to the one of \cite[Lemma 4]{Sil19}. For the reader’s convenience, we shall report it.

\begin{lem}\label{lem:gaugefixing} For every $\omega \in W^{1,p}(B_r,\Lambda^\ell)$, there is $\tilde\omega \in W^{1,p}(B_r,\Lambda^\ell)$ such that 
$$
d \tilde \omega =d  \omega 
\quad\text{and}\quad
d^* \tilde \omega =0.
$$
\end{lem}
\begin{proof}
If $\ell=0$, then $d^*\omega=0$ by definition, hence choose $\tilde\omega=\omega$. Suppose $\ell\ge 1$. Then, since $d^*\omega\in L^{p}(B_r,\Lambda^{\ell-1})$ and $d^*(d^*\omega)=0$, there exists a unique solution $\theta\in W^{2,p}(B_r,\Lambda^{\ell-1})$ to 
$$
\begin{cases}
d^* (d\theta) = d^* \omega & \quad \text{in }\ B_r,\\
d^* \theta =0 & \quad \text{in }\ B_r,\\
\nu \wedge \theta =0 & \quad \text{on }\ \partial B_r,
\end{cases}
$$
see for instance \cite[Theorem 9]{Sil17}. Therefore, we complete the proof by choosing $\tilde \omega = \omega - d \theta$.
\end{proof}

\subsection{Sobolev-Poincar\'e inequality}

We introduce Sobolev-Poincar\'e type inequality for differential forms, for which we refer to results in \cite[Section 4]{IL93}. Let $0\le \ell \le n-1$ and $1<p<\infty$. We define the operator $T:L^p(B_r,\Lambda^{\ell+1})\to W^{1,p}(B_r,\Lambda^{\ell})$ as in \cite[Proposition 4.1]{IL93} with $D=B_r$. Note that from  \cite[Lemma 4.2]{IL93} we have 
$$
\omega = d (T\omega) + T(d\omega) , \quad \omega\in W^{d,p}(B_r,\Lambda^{\ell+1}).
$$
For $\omega \in W^{d,p}(B_r,\Lambda^{\ell})$, define
\begin{equation}\label{omegaD}
\omega_{B_r} := \omega- T(d\omega).
\end{equation}
Then we have 
$$
d \omega_{B_r} = dd (T\omega) + d(T(d\omega)) - d(T(d\omega)) =0 .
$$
Note that if $\omega$ is a $0$-form, then $\omega_{B_r}$ is the average of $\omega$ in $B_r$ given by \eqref{Ave}.
Now we recall the Sobolev-Poincar\'e inequality for differential forms in \cite[Corollary 4.2]{IL93}.
\begin{lem}\label{lem:sobopoin}
Let $\omega \in W^{1,p}(B_r,\Lambda^{\ell})$ with $1<p< n$ and $0\le \ell \le n-1$, and $\omega_{B_r}  \in W^{1,p}(B_r,\Lambda^{\ell})$. 
Then there exists $c=c(n,\ell,p)>0$ such that
$$
\|\omega-\omega_{B_r}\|_{L^{\frac{np}{n-p}}(B_r)}  \le c \|d \omega \|_{L^p(B_r)}.
$$
\end{lem}

\subsection{Regularity results in constant coefficient case}
We consider the weak solution  $h\in W^{d,p}(B_{2r},\Lambda^{\ell})$ to
\begin{equation}\label{pLaplaceeq}
d^*(|dh|^{p-2}dh) =0 \quad \text{in }\ B_{r}.
\end{equation}
Note that this equation is exactly the same as \eqref{mainpb} with $F\equiv 0$ and $a\equiv 1$. Then the results in \cite{Uhl77} and \cite{Ham92} imply that $dh\in C^{0,\alpha}_{\loc}(B_{2r},\Lambda^{\ell+1})$   for some $\alpha\in(0,1)$ depending on $n$ and $p$. In this paper, we will  only use the local boundedness of $dh$. In particular, we refer to Theorem~4.1 in \cite{Ham92}.
\begin{lem}\label{lem:locbdd_h}
Let $h\in W^{d,p}(B_{r},\Lambda^{\ell})$ be a weak solution to
\eqref{pLaplaceeq}. Then we have
 $$
 \sup_{B_{r/2}} |d h | \le c \left(\fint_{B_r}|dh|^p\,dx\right)^{\frac{1}{p}}
 $$
 for some $c=c(n,\ell,p)>0$.
\end{lem}

\section{Calder\'on-Zygmund estimates}\label{Sec_last}

In this section we prove our main theorem, Theorem~\ref{mainthm}. We start by deriving two lemmas: the first is \textit{higher integrability}, sometimes referred to as the \textit{self-improving property}, and the second is \textit{comparison estimates}.

\subsection{Higher integrability}
We first derive higher integrability for the exterior derivatives of  weak solutions to \eqref{mainpb} with general coefficients. 
\begin{lem}\label{lem:high}
Let  $u\in W^{1,p}_{\loc}(\Omega, \Lambda^{\ell})$ be a local weak solution to  \eqref{mainpb}. If $F\in L^\frac{q}{p-1}(\Omega,\Lambda^{\ell +1})$ for some $q>p$, then there exists $\sigma=\sigma(n,\ell,\nu,L,p,q)$ with $0<\sigma<\frac{q}{p}-1$ such that $|du|^p \in L^{1+\sigma}_{\loc}(\Omega)$ with the estimate
$$
\fint_{B_r} |du|^{p(1+\sigma)}\,dx  \le c \left(\fint_{B_{2r}} |du|^{p}\,dx \right)^{1+\sigma} + c \fint_{B_{2r}} |F|^{\frac{p}{p-1}(1+\sigma)}\,dx
$$
for all $B_{2r}\Subset\Omega$ and for some $c=c(n,\ell,\nu,L,p,q)>0$.
\end{lem}

\begin{proof}
For  $B_{2r}\Subset \Omega$, let $\eta\in C^\infty_0(B_{2r})$ be a standard cut-off function such that $0\le \eta \le 1$, $\eta\equiv 1$ in $B_r$ and $|\nabla \eta|\le c(n)/r$. We also consider  the closed form  $u_{B_{2r}}\in W^{1,p}(B_{2r},\Lambda^{\ell})$ defined as in \eqref{omegaD} with $B_{2r}$ in place of $B_r$. Note that $d u_{B_{2r}}=0$. Then we take $\eta^p(u-u_{B_{2r}})\in W^{1,p}_0(B_{2r},\Lambda^{\ell})$ as a test function in the weak formulation \eqref{weakform} to obtain
$$\begin{aligned}
0 & = \int_{B_{2r}} a(x)|du|^{p-2} \langle du , d(\eta^p(u-u_{B_{2r}})\rangle \, dx - \int_{B_{2r}}  \langle F , d(\eta^p(u-u_{B_{2r}})\rangle \, dx \\
& = \int_{B_{2r}} a(x)|du|^{p-2} \eta^p \langle du , du\rangle  \, dx + \int_{B_{2r}} a(x)|du|^{p-2} p \eta^{p-1}\langle du , d\eta \wedge (u-u_{B_{2r}})\rangle \, dx \\
& \qquad - \int_{B_{2r}} \eta^p \langle F , du\rangle \, dx  - \int_{B_{2r}} p\eta^{p-1} \langle F , d\eta \wedge (u-u_{B_{2r}})\rangle \, dx,
\end{aligned}$$
which implies that
$$\begin{aligned}
I_1&:= \int_{B_{2r}} a(x)|du|^{p-2} \eta^p \langle du , du\rangle  \, dx \\
 &=- \int_{B_{2r}} a(x)|du|^{p-2} p \eta^{p-1}\langle du , d\eta \wedge (u-u_{B_{2r}})\rangle \, dx \\
& \quad + \int_{B_{2r}} \eta^p \langle F , du\rangle \, dx  + \int_{B_{2r}} p\eta^{p-1} \langle F , d\eta \wedge (u-u_{B_{2r}})\rangle \, dx=: I_2 + I_3+I_4.
\end{aligned}$$
We estimate $I_1 \sim I_4$ separately.
For $I_1$ and $I_2$,  since $0<\nu \le a(x)\le L$ for any $x \in \Omega$, we derive
$$
I_1 \ge \nu \int_{B_{2r}} |du|^{p} \eta^p   \, dx
$$
and
$$\begin{aligned}
|I_2|& \le  L p \int_{B_{2r}} |du|^{p-1}  \eta^{p-1} |d\eta| |u-u_{B_{2r}}| \, dx\\
& \le  c \int_{B_{2r}} |du|^{p-1}  \eta^{p-1} \frac{|u-u_{B_{2r}}|}{2r} \, dx\\
& \le   \frac{\nu}{4} \int_{B_{2r}} |du|^{p} \eta^p   \, dx + c \int_{B_{2r}}  \left[\frac{|u-u_{B_{2r}}|}{2r}\right]^p \, dx
\end{aligned}$$
 by using Young's  inequality. 
 Similarly, we obtain 
$$\begin{aligned}
|I_3| \le   \int_{B_{2r}} \eta^p |F||du| \, dx  \le  \frac{\nu}{4} \int_{B_{2r}} |du|^{p} \eta^p   \, dx + c \int_{B_{2r}}  |F|^{\frac{p}{p-1}} \eta^p \, dx
\end{aligned}$$
and
$$\begin{aligned}
|I_4|& \le   c \int_{B_{2r}} |F|\eta^{p-1} \frac{|u-u_{B_{2r}}|}{2r} \, dx \le  c \int_{B_{2r}} |F|^{\frac{p}{p-1}} \eta^p   \, dx + c \int_{B_{2r}}  \left[\frac{|u-u_{B_{2r}}|}{2r}\right]^p \, dx.
\end{aligned}$$

Combining the above results, we conclude
$$
\fint_{B_{r}} |du|^{p}   \, dx  \le  2^n \fint_{B_{2r}} |du|^{p} \eta^p   \, dx   \le c \fint_{B_{2r}}  \left[\frac{|u-u_{B_{2r}}|}{2r}\right]^p \, dx + c \fint_{B_{2r}} |F|^{\frac{p}{p-1}}   \, dx .
$$ 
Here, by the Sobolev-Poincar\'e inequality in Lemma~\ref{lem:sobopoin} with $p$ and $B_r$ replaced by 
$p_* = \min\{\frac{p+1}{2}, \frac{np}{n+p} \}  \in (1,p) $ and $B_{2r}$, respectively, we have
$$ 
\fint_{B_{2r}}  \left[\frac{|u-u_{B_{2r}}|}{2r}\right]^p \, dx \le c \left(\fint_{B_{2r}} |du|^{p_*}\,dx \right)^{\frac{p}{p_*}} .
$$
Therefore, we obtain
$$
\fint_{B_{r}} |du|^{p}   \, dx \le c \left(\fint_{B_{2r}} |du|^{p_*}\,dx \right)^{\frac{p}{p_*}} + c \fint_{B_{2r}} |F|^{\frac{p}{p-1}}   \, dx
$$
for every $B_{2r}\Subset \Omega$. Finally, applying Lemma~\ref{lem:Gehring} below with $f=|du|^p$,  $g=|F|^{\frac{p}{p-1}}$, $m=\frac{p_*}{p}$ and $s=q/p$, we get the conclusion.  
\end{proof}

\begin{lem} \label{lem:Gehring}(Gehring, \cite[Theorem 6.6]{Giusti})  
Let $f\in L^1(B_r)$ and $g\in L^s(B_r)$ for some $s>1,$ and assume that for every $B_{\rho}(y)\subset B_{2\rho}(y)\subset B_r$ we have 
$$
\fint_{B_{\rho}(y)} |f|   \, dx \le C \left(\fint_{B_{2\rho}(y)} |f|^{m}\,dx \right)^{\frac{1}{m}} +C \fint_{B_{2\rho}(y)} |g|   \, dx
$$
with $0<m<1$. Then there exists $\sigma=\sigma(n,C,m,s)$ with $0< \sigma< s-1$ such that $f\in L^{1+\sigma}(B_{r/2})$ and 
$$
\fint_{B_{r/2}} |f|^{1+\sigma}   \, dx \le c \left(\fint_{B_{r}} |f| \,dx \right)^{1+\sigma} + \fint_{B_{r}} |g|^{1+\sigma}   \, dx
$$
for some $c=c(n,C,m)>0$.
\end{lem}

\subsection{Comparison estimates}
Let $u\in W^{1,p}_{\loc}(\Omega, \Lambda^\ell)$ be a weak solution to \eqref{mainpb}, where $a:\Omega\to [\nu,L]$ satisfies the $(\delta,R)$-vanishing condition for some $R>0$ and $F\in L^{\frac{q}{p-1}}(\Omega,\Lambda^{\ell+1})$ for some $q>p$. Note that $\delta\in(0,1)$ will be chosen sufficiently small later in Section~\ref{sec:proof}.  We fix any $B_{4r}\Subset \Omega$ with $r\le R/2$. 
We first note from the higher integrability in Lemma~\ref{lem:high} with $r$ replaced by $2r$ that 
\begin{equation}\label{highestimate}
\left(\fint_{B_{2r}} |du|^{p(1+\sigma)} \, dx \right)^{\frac{1}{1+\sigma}} \le c\fint_{B_{4r}} |du|^{p} \, dx + c\left(\fint_{B_{4r}} |F|^{\frac{p}{p-1}(1+\sigma)} \, dx \right)^{\frac{1}{1+\sigma}}.
\end{equation}
Moreover, in view of Lemma~\ref{lem:gaugefixing}, we can assume that 
$$
d^* u=0 \quad  \text{in }\  B_{2r}.
$$

Let  
$h\in u+ W^{1,p}_{d^*,T}(B_{2r},\Lambda^\ell)$ be the unique weak solution to 
\begin{equation}\label{eqcom}
\begin{cases}
d^*(\overline a |dh|^{p-2}d h) =0  & \text{in }\ B_{2r},\\
d^*h = d^* u =0  & \text{in }\ B_{2r},\\
\nu \wedge  h  = \nu \wedge u & \text{on } \ \partial B_{2r},
\end{cases}
\end{equation}
where 
$$
\overline a = (a)_{B_{2r}} = \fint_{B_{2r}} a(x)\, dx.
$$
Then we have the following comparison estimates.

\begin{lem}\label{lem:comparision}
Under the above setting, suppose 
$$
\fint_{B_{4r}} |du|^{p} \, dx  \le \lambda 
$$
for some $\lambda>0$.
Then we have that
\begin{equation}\label{comparison0}
\fint_{B_{2r}} |dh|^p \,dx \le  \fint_{B_{2r}} |du|^p \,dx  \le 2^n \lambda.
\end{equation}
Moreover, for every $\varepsilon\in (0,1)$, there exist small $\delta,\delta_1\in(0,1)$ depending on $n,\ell,p,\nu,L, q$ and $\varepsilon$  such that if  $a$ is $(\delta,R)$-vanishing and 
\begin{equation}\label{Fdelta1}
\left(\fint_{B_{4r}} |F|^{\frac{p}{p-1}(1+\sigma)} \, dx \right)^{\frac{1}{1+\sigma}} \le \delta_1 \lambda,
\end{equation}
then
\begin{equation}\label{comparison1}
\fint_{B_{2r}} |du -dh|^p \,dx \le \varepsilon \lambda .
\end{equation}
\end{lem}

\begin{proof}
We first obtain \eqref{comparison0}. Since  $u-h \in  W^{1,p}_{d^*,T}(B_{2r},\Lambda^{\ell})$, we take $u-h$ as the test function in the weak formulations of \eqref{mainpb} and \eqref{eqcom}, see Remark~\ref{rmk:weakform}, to discover
\begin{equation}\label{lem:comparison_pf1}
\int_{B_{2r}} a(x) |du|^{p-2} \langle du , du -dh\rangle \,dx = \int_{B_{2r}}  \langle F , du -dh\rangle \,dx  
\end{equation}
and
\begin{equation}\label{lem:comparison_pf2}
\int_{B_{2r}} \overline a |dh|^{p-2} \langle dh , du -dh\rangle \,dx = 0  .
\end{equation}
We first notice from \eqref{lem:comparison_pf2} and Young's inequality that
$$\begin{aligned}
\int_{B_{2r}} |dh|^{p}  \,dx  & = \int_{B_{2r}} |dh|^{p-2} \langle dh , du \rangle \, dx  \\
&\le \int_{B_{2r}} |dh|^{p-1}  |du| \, dx   \le  \frac{p-1}{p}   \int_{B_{2r}} |dh|^{p}  \,dx +   \frac1p  \int_{B_{2r}} |du|^{p}  \,dx,
\end{aligned}$$
which implies \eqref{comparison0}.

We next obtain \eqref{comparison1}. Combining \eqref{lem:comparison_pf1} and \eqref{lem:comparison_pf2}, we obtain
\begin{equation}\label{lem:comparison_pf3}\begin{aligned}
I_5 & : = \fint_{B_{2r}} \overline a  \langle |du|^{p-2} du -|dh|^{p-2} dh , du -dh\rangle \,dx \\
&\ =  \fint_{B_{2r}} ( \overline a- a(x)) |du|^{p-2} \langle  du  , du -dh \rangle \,dx   + \fint_{B_{2r}}  \langle F , du -dh\rangle \,dx  =: I_6+I_7.
\end{aligned}\end{equation}
We estimate $I_5 \sim I_7$ separately.

Recall the elementary inequality
\[
( |a|^{p-2}a - |b|^{p-2}b)\cdot(a-b) \ge c(p,m)\, (|a|^2+|b|^2)^{\frac{p-2}{2}}|a-b|^2
\]
for any $a, b \in \mathbb{R}^m$ and any $p>1$.
Then for $I_5$, 
 we have 
$$
 \fint_{B_{2r}} (|du|^2 + |dh|^2)^{\frac{p-2}{2}} |du - dh|^2 \,dx \le c I_5.
$$
Hence if $p\ge 2$,
$$
\fint_{B_{2r}}  |du - dh|^p \,dx \le c I_5,
$$
and if $1<p<2$, by Young's inequality 
and \eqref{comparison1}, for every $\epsilon_1\in(0,1)$,
$$\begin{aligned}
& \fint_{B_{2r}}  |du - dh|^p \,dx  = \fint_{B_{2r}} (|du|^2 + |dh|^2)^{\frac{p(2-p)}{4}} (|du|^2 + |dh|^2)^{\frac{p(p-2)}{4}} |du - dh|^p \,dx \\
 &\qquad \le  \epsilon_1 \fint_{B_{2r}} (|du|^2 + |dh|^2)^{\frac{p}{2}} \, dx +c \epsilon_1^{-\frac{2-p}{p}} \fint_{B_{2r}}  (|du|^2 + |dh|^2)^{\frac{p-2}{2}} |du - dh|^2 \,dx \\
 &\qquad \le  c \epsilon_1 \lambda + c\epsilon_1^{-\frac{2-p}{p}} I_5.
\end{aligned}$$

For $I_6$,  by Young's inequality, H\"older's inequality and the inequality $0<\nu \le a \le L$, we have that for every $\epsilon_2\in (0,1)$,
$$\begin{aligned}
&|I_6| \le \fint_{B_{2r}}|a(x)-\overline a| |du|^{p-1} |du-dh|\, dx \\
& \le \epsilon_2   \fint_{B_{2r}} |du-dh|^p \, dx +  c\epsilon_2^{-\frac{1}{p-1}} \fint_{B_{2r}}|a(x)-\overline a|^\frac{p}{p-1} |du|^{p} \, dx \\
& \le \epsilon_2   \fint_{B_{2r}} |du-dh|^p \, dx +   c\epsilon_2^{-\frac{1}{p-1}} \left(\fint_{B_{2r}}|a(x)-\overline a|^{\frac{p(1+\sigma)}{(p-1)\sigma}}  \,dx \right)^{\frac{\sigma}{1+\sigma}} \left(\fint_{B_{2r}} |du|^{p(1+\sigma)} \, dx \right)^{\frac{1}{1+\sigma}}\\
& \le \epsilon_2   \fint_{B_{2r}} |du-dh|^p \, dx +  c \epsilon_2^{-\frac{1}{p-1}} \left(\fint_{B_{2r}}|a(x)-\overline a|  \,dx \right)^{\frac{\sigma}{1+\sigma}} \left(\fint_{B_{2r}} |du|^{p(1+\sigma)} \, dx \right)^{\frac{1}{1+\sigma}}.
\end{aligned}$$
Then applying the $(\delta, R)$-vanishing condition of $a$, \eqref{highestimate} and \eqref{Fdelta1}, we have
$$
|I_6| \le  \epsilon_2   \fint_{B_{2r}} |du-dh|^p \, dx +   c \epsilon_2^{-\frac{1}{p-1}} \delta^{\frac{\sigma}{1+\sigma}} (1+\delta_1)\lambda \le \epsilon_2   \fint_{B_{2r}} |du-dh|^p \, dx +   c \epsilon_2^{-\frac{1}{p-1}} \delta^{\frac{\sigma}{1+\sigma}} \lambda.
$$
For $I_7$,  by Young's inequality with \eqref{Fdelta1}, we have that for every $\epsilon_3\in(0,1)$,
$$
|I_7| \le \fint_{B_{2r}} |F| |du -dh| \,dx \le \epsilon_3 \fint_{B_{2r}} |du -dh|^p \,dx + c \epsilon_3^{-\frac{1}{p-1}} \delta_1\lambda.
$$
Therefore, inserting the resulting estimates for $I_5 \sim I_7$ into \eqref{lem:comparison_pf3}, we have that for every $\epsilon_1,\epsilon_2, \epsilon_3 \in(0,1)$,
$$\begin{aligned}
\fint_{B_{2r}} |du-dh|^p \, dx &\le c\epsilon_1^{-\frac{2-p}{p}} \epsilon_2   \fint_{B_{2r}} |du-dh|^p \, dx  +c\epsilon_1^{-\frac{2-p}{p}}  \epsilon_3 \fint_{B_{2r}} |du -dh|^p \,dx \\
& \qquad + \left( c\epsilon_1^{-\frac{2-p}{p}}   \epsilon_2^{-\frac{1}{p-1}} \delta^{\frac{\sigma}{1+\sigma}}  + c \epsilon_1^{-\frac{2-p}{p}}   \epsilon_3^{-\frac{1}{p-1}} \delta_1\right)\lambda +  c \epsilon_1 \lambda .
\end{aligned}$$
Consequently, choosing first $\epsilon_1$ such that  $c\epsilon_1=\frac{\varepsilon}{6}$, second  $\epsilon_2$ and $\epsilon_3$ such that  $c\epsilon_1^{-\frac{2-p}{p}} \epsilon_2 =c\epsilon_1^{-\frac{2-p}{p}}  \epsilon_3=\frac{1}{4}$, and third $\delta$ and $\delta_1$ such that $c\epsilon_1^{-\frac{2-p}{p}}   \epsilon_2^{-\frac{1}{p-1}} \delta^{\frac{\sigma}{1+\sigma}}  = c \epsilon_1^{-\frac{2-p}{p}}   \epsilon_3^{-\frac{1}{p-1}} \delta_1 =\frac{\varepsilon}{6}$, we obtain \eqref{comparison1}.

\end{proof}

\subsection{Proof of Theorem~\ref{mainthm}}\label{sec:proof}

We adopt the standard approach introduced by Acerbi and Mingione in \cite{AM07}. Nevertheless, we shall provide a detailed proof for completeness.

Let $0< r \le R/2$ and $B_{2r}\Subset\Omega$. We
define
\begin{equation}\label{def_lambda0}
\lambda_0 :=   \fint_{B_{2r}} |du|^{p} \, dx + \frac{1}{
\delta_1}\bigg(\fint_{B_{2r}} |F|^{\frac{p}{p-1}(1+\sigma)} \, dx \bigg)^{\frac{1}{1+\sigma}} \ge 0,
\end{equation} 
where $\sigma$ is given in Lemma~\ref{lem:high} and a small $\delta_1>0$ will be selected later. 
We denote the super-level set 
\[
E(\lambda, \rho ) := \{ y \in B_{\rho} : |du(y)|^p > \lambda  \}
\]
for $\lambda>0$ and $\rho >0$, where the center of $B_{\rho}$ is the same as the one of $B_{2r}$.

Now fix any $s_1, s_2$ with $1\le s_1 < s_2 \le 2$, and consider any $\lambda$ satisfying 
\begin{equation}\label{defA}
\lambda >  A\lambda_0 , 
\qquad  \text{where }\ A:=   \frac{20^n}{(s_2-s_1)^n}.
\end{equation}
For $y \in E(\lambda,s_1r)$ and $\rho \in \big[\frac{(s_2-s_1)r}{10}, (s_2-s_1)r\big]$, we derive that
$$\begin{aligned}
 \fint_{B_{\rho}(y)} |du|^p\,dx & + \frac{1}{
\delta_1}\left(\fint_{B_{\rho}(y)} |F|^{\frac{p}{p-1}(1+\sigma)} \, dx \right)^{\frac{1}{1+\sigma}} \\
& \le \frac{|B_{2r}|}{|B_{\rho}(y)|} \bigg[ \fint_{B_{2r}} |du|^p\,dx +\frac{1}{
\delta_1}\bigg(\fint_{B_{2r}} |F|^{\frac{p}{p-1}(1+\sigma)} \, dx\bigg)^{\frac{1}{1+\sigma}} \bigg]\\
& \le \bigg(\frac{2r}{\rho}\bigg)^n \lambda_0 \le  \frac{20^n  \lambda_0 }{(s_2-s_1)^n}  \le \lambda \end{aligned}$$
 from \eqref{def_lambda0}.
On the other hand, we apply the Lebesgue differentiation theorem to obtain that for almost every $y \in E(\lambda, s_1r)$
\[
\lim_{\rho\rightarrow 0 } \left[  \fint_{B_{\rho}(y)} |du|^p\,dx + \frac{1}{
\delta_1} \left(\fint_{B_{\rho}(y)} |F|^{\frac{p}{p-1}(1+\sigma)} \, dx \right)^{\frac{1}{1+\sigma}} \right] > \lambda.
\]
Then by the continuity of integral, we see that for almost every $y \in E(\lambda, s_1r)$,
there exists $\rho_y  \in (0, \frac{(s_2-s_1)r}{10})$ such that
\[
 \fint_{B_{\rho_y}(y)} |du|^p\,dx + \frac{1}{
\delta_1}\left(\fint_{B_{\rho_y}(y)} |F|^{\frac{p}{p-1}(1+\sigma)} \, dx \right)^{\frac{1}{1+\sigma}}  =\lambda
\]
and 
\[
 \fint_{B_{\rho}(y)} |du|^p\,dx + \frac{1}{
\delta_1}\left(\fint_{B_{\rho}(y)} |F|^{\frac{p}{p-1}(1+\sigma)} \, dx \right)^{\frac{1}{1+\sigma}} < \lambda
\]
for all $\rho \in (\rho_y, (s_2-s_1)r]$.

Therefore, for given $\lambda >  A \lambda_0$,
by virtue of Vitali covering lemma, 
 there is a disjoint collection of balls $\{ B_{\rho_i}(y^i) \}_{i=1}^{\infty}$ with $y^i \in E(\lambda, s_1r)$ and $\rho_i \in (0, \frac{(s_2-s_1)r}{20})$ such that 
 \begin{equation}\label{coverings}
E(\lambda, s_1) \setminus \mathcal{N} \subset \bigcup_{i=1}^{\infty} B_{5\rho_i}(y^i) \subset B_{2r},
\end{equation}
where $\mathcal{N}$ is a negligible set, 
\begin{equation}\label{eq_lambda}
 \fint_{B_{\rho_i}(y^i)} |du|^p\,dx + \frac{1}{
\delta_1} \left(\fint_{B_{\rho_i}(y^i)} |F|^{\frac{p}{p-1}(1+\sigma)} \, dx\right)^{\frac{1}{1+\sigma}} =\lambda
\end{equation}
and 
\[
 \fint_{B_{\rho}(y^i)} |du|^p\,dx + \frac{1}{
\delta_1} \left(\fint_{B_{\rho}(y^i)} |F|^{\frac{p}{p-1}(1+\sigma)} \, dx\right)^{\frac{1}{1+\sigma}} <\lambda
\]
for all $\rho \in (\rho_i, (s_2-s_1)r]$.

We first note that 
\[
 \fint_{B_{20\varrho_i}(y^i)} |du|^p\,dx  <\lambda  
\quad\text{and}\quad
 \left(\fint_{B_{20\varrho_i}(y^i)} |F|^{\frac{p}{p-1}(1+\sigma)} \, dx\right)^{\frac{1}{1+\sigma}} <\delta_1\lambda.
\]
Let $\varepsilon\in (0,1)$. Then applying Lemma~\ref{lem:comparision} there exist small $\delta_1>0$ depending on $n,\ell,p,\nu,\Lambda, q$ and $\varepsilon$  such that
\begin{equation}\label{comparisoni}
\fint_{B_{5\rho_i}(y^i)} |du -dh_i|^p \,dx \le \varepsilon \lambda,
\end{equation}
where $h_i \in W^{1,p}_{d^*,T}(B_{10\rho_i}(y^i))$ is the unique weak solution to 
\[\begin{cases}
d^*(\overline a |dh_i|^{p-2}d h_i) =0  & \text{in }\ B_{10\varrho_i}(y^i),\\
d^*h_i = d^* u =0  & \text{in }\ B_{10\varrho_i}(y^i),\\
\nu \wedge  h_i  = \nu \wedge u & \text{on } \ \partial B_{10\varrho_i}(y^i).
\end{cases}\]
Moreover, by Lemma~\ref{lem:locbdd_h}, we get
\begin{equation}\label{supi}
 \sup_{B_{5\rho_i}(y^i)} |d h_i |^p \le c\fint_{B_{10\rho_i}(y^i)}|dh_i|^p\,dx \le c_0\lambda
\end{equation}
for some $c_0=c_0(n,\ell,p) \ge 1 .$

Now we  write
\[ E_i(\tilde\lambda) := \{ x \in B_{\rho_i}(y^i) : |du(x)|^p > \tilde\lambda  \}
\]
and
\[
H_i(\tilde\lambda):=\{ x \in B_{\rho_i}(y^i) : |F(x)|^{\frac{p}{p-1}} > \tilde\lambda  \}
\]
for $\tilde\lambda>0$.
Then from \eqref{eq_lambda}, we see 
$$
\fint_{B_{\rho_i}(y^i)} |du|^p\,dx \ge \frac{\lambda}{2} 
\quad\text{or}\quad
\frac{1}{
\delta_1} \left(\fint_{B_{\rho_i}(y^i)} |F|^{\frac{p}{p-1}(1+\sigma)} \, dx\right)^{\frac{1}{1+\sigma}} \ge \frac{\lambda}{2},
$$
hence
 $$
|B_{\rho_i}(y^i)| \le \frac{2}{\lambda} \int_{B_{\rho_i}(y^i)} |du|^p\,dx \le \frac{2}{\lambda} \int_{E_i(\lambda/4)} |du|^p\,dx  + \frac{1}{2} |B_{\rho_i}(y^i)| 
$$
or
 $$\begin{aligned}
|B_{\rho_i}(y^i)| & \le \left(\frac{2}{\delta_1\lambda}\right)^{1+\sigma}\int_{B_{\rho}(y^i)} |F|^{\frac{p}{p-1}(1+\sigma)} \, dx \\
&\le \left(\frac{2}{\delta_1\lambda}\right)^{1+\sigma}\int_{H_i(\delta_1\lambda/4)}  |F|^{\frac{p}{p-1}(1+\sigma)} \, dx  + \frac{1}{2^{1+\sigma}} |B_{\rho_i}(y^i)|  
\end{aligned}$$
and in turn, we obtain that 
\begin{equation}\label{ineq_measure_Bi}
|B_{\rho_i}(y^i)| \le 4^{1+\sigma}\bigg(\frac{1}{\lambda}\int_{E_i(\lambda/4)} |du|^p\,dx +\frac{1}{(\delta_1\lambda)^{1+\sigma}}\int_{H_i(\delta_1\lambda/4)} |F|^{\frac{p}{p-1}(1+\sigma)} \, dx \bigg).
\end{equation}

Set 
\begin{equation}\label{defK}
K := 2^p c_0 >1,
\end{equation}
where the constant $c_0$ is from \eqref{supi}.
From \eqref{coverings}, we then have
\[
E(K\lambda, s_1r) \setminus \mathcal{N} \subset E(\lambda, s_1r) \setminus \mathcal{N}\subset \bigcup_{i=1}^{\infty} B_{5\rho_i}(y^i) \subset B_{s_2r}.
\]
For any $x \in B_{5\rho_i}(y^i)$ such that $|du(x)|^p>K\lambda$,
we infer from \eqref{supi} that
 $$\begin{aligned}
 |du(x)|^p &\le 2^{p-1} (|du(x)- dh_i(x)|^p + |dh_i(x)|^p) \\
 & \le 2^{p-1} |du(x)- dh_i(x)|^p +2^{p-1}c_0 \lambda \\
 & = 2^{p-1} |du(x)- dh_i(x)|^p +\frac{K\lambda}{2}\\  
 & <2^{p-1} |du(x)- dh_i(x)|^p  + \frac12   |du(x)|^p,
 \end{aligned}$$
which means 
 \[
 |du(x)|^p \le 2^p |du(x)- dh_i(x)|^p.
 \]
 Consequently, using the previous results, \eqref{comparisoni} and \eqref{ineq_measure_Bi} we obtain that 
\begin{equation}\label{est_result}
\begin{aligned}
\int_{E(K\lambda, s_1r)} |du|^p \,dx 
&\le \sum_{i} \int_{B_{5\rho_i}(y^i) \cap E(K\lambda, s_1r) } |du|^p \,dx  \\
& \le 2^p  \sum_{i} \int_{B_{5\rho_i}(y^i) \cap E(K\lambda, s_1r) } |du- dh_i|^p \,dx \\
& \le c \varepsilon \lambda \sum_{i}  |B_{\rho_i}(y^i)| \\
& \le c \varepsilon \bigg( \int_{B_{s_2r} \cap \{ |du|^p>\lambda/4\}} |du|^p\,dx\\
&\quad\qquad +\frac1{\delta_1^{1+\sigma}\lambda^\sigma} \int_{B_{s_2r} \cap \{ |F|^{\frac{p}{p-1}}>\delta_1\lambda/4\} } |F|^{\frac{p}{p-1}(1+\sigma)} \, dx\bigg).
\end{aligned}
\end{equation}

Let us define the truncated function of $|du|$ with respect to the level $k>0$ as
$$
|du|^p_k:=\min\{|du|^p,k\}.
$$
In addition, $k$ is assumed to be sufficiently large so that $k> K A \lambda$, where $A$ and $K$ are defined in \eqref{defA} and \eqref{defK}.
Now we estimate the following integral
$$
\int_{B_{s_1r}}|du|_k^{q-p}|du|^{p}\, dx.
$$
For the sake of simplicity, let us denote the super-level set with respect to $|du|_k$ by 
$$ E_{k}(K\lambda, s_1r):= \left\{ x \in B_{s_1r} : |du(x)|^p_k > K\lambda  \right\}.$$
We notice that when $k> K\lambda,$ $\{ |du|^p_k > K\lambda \} = \{ |du|^p>K\lambda\}$ and $\{ |du|^p_k > \lambda/4\} =\{ |du|^p > \lambda/4\}$. On the other hand,  when  $k\leq K\lambda$, it is clear that $\{ |du|^p_k > K\lambda \} =  \emptyset$. By \eqref{est_result}, it turns out that 
\begin{equation}\label{est_duEk}\begin{aligned}
 \int_{E_k(K\lambda, s_1r)} |du|^p \;dx 
  &\le  c \varepsilon  \bigg( \int_{B_{s_2r}\cap \{ |du|^p_k> \lambda/4 \}} |du|^p\,dx \\
&   \qquad \qquad + \frac{1}{\delta_1^{1+\sigma}\lambda^\sigma} \int_{B_{s_2r}\cap \{ |F|^{\frac{p}{p-1}} > \delta_1\lambda/4 \} } |F|^{\frac{p}{p-1}(1+\sigma)} \,dx \bigg).
\end{aligned}\end{equation}
We multiply both sides of \eqref{est_duEk} by $\lambda^{\frac{q-p}{p}-1}$ and integrate with respect to $\lambda$ over $(A\lambda_0,\infty)$ to discover
\begin{equation}\label{est_duEkm}
\begin{aligned}
 J_0 & := \int_{A\lambda_0}^{\infty} \lambda^{\frac{q-p}{p}-1}\int_{E_k(K\lambda, s_1 r)} |du|^p \;dxd\lambda \\
& \le  c\varepsilon  \int_{A\lambda_0}^{\infty} \lambda^{\frac{q-p}{p}-1} \bigg( \int_{B_{s_2r}\cap \{ |du|^p_k> \lambda/4 \}} |du|^p\,dx \bigg) d\lambda \\
&\qquad \quad + \frac{c \varepsilon}{\delta_1^{1+\sigma}}  \int_{A\lambda_0}^{\infty} \lambda^{\frac{q-p}{p}-1-\sigma} \bigg(
 \int_{B_{s_2r}\cap \{ |F|^{\frac{p}{p-1}} > \delta_1\lambda/4\} } |F|^{\frac{p}{p-1}(1+\sigma)} \,dx \bigg) d \lambda\\
 & \quad =: c \varepsilon J_1 +\frac{c \varepsilon}{\delta_1^{1+\sigma}} J_2.
\end{aligned}\end{equation}
By virtue of Fubini's theorem, we obtain that 
\begin{equation}\label{est_duEkI_0}
\begin{aligned}
 J_0 &:= \int_{A\lambda_0}^{\infty} \lambda^{\frac{q-p}{p}-1}\int_{E_k(K\lambda, s_1r)} |du|^p \;dxd\lambda\\
& = \int_{E_k(KA\lambda_0, s_1 r)} |du|^p \left[ \int_{A\lambda_0}^{|du|^p_k / K} \lambda^{\frac{q-p}{p}-1} \;d\lambda \right] dx \\
  &=  \frac{p}{q-p} \int_{E_k(KA\lambda_0, s_1r)} |du|^p \left( \frac{|du|_k^{q-p}}{K^{\frac{q-p}{p}}}- (A\lambda_0)^{\frac{q-p}{p}}\right) dx.
\end{aligned}\end{equation}
 We again use Fubini's theorem to derive that 
\begin{equation}\label{est_duEkI_1}
\begin{aligned}
 J_1&:=  \int_{A\lambda_0}^{\infty} \lambda^{\frac{q-p}{p}-1} \Bigg( \int_{B_{s_2r}\cap \{ |du|^p_k> \lambda/4\}} |du|^p\,dx \Bigg) d\lambda \\
&=  \int_{B_{s_2r}\cap \{ |du|^p_k> A\lambda_0 /4 \}} |du|^p \left( \int_{A\lambda_0}^{4|du|^p_k}  \lambda^{\frac{q-p}{p}-1}      d\lambda\right) dx\\
& =  \frac{p}{q-p}  \int_{B_{s_2r}\cap \{ |du|^p_k> A\lambda_0 /4 \}} |du|^p \left( 4^{\frac{q-p}{p}} |du|_k^{q-p} -[A\lambda_0]^{\frac{q-p}{p}} \right) dx\\
&\le  \frac{p\, 4^{\frac{q-p}{p}}}{q-p}\int_{B_{s_2r}} |du|^p |du|_k^{q-p} dx. 
 \end{aligned}\end{equation}
 Similarly, we infer that 
 \begin{equation}\label{est_duEkI_2}
\begin{aligned}
 J_2&:=  \int_{A\lambda_0}^{\infty} \lambda^{\frac{q-p}{p}-1-\sigma} \Bigg( \int_{B_{s_2r}\cap \{ |F|^{\frac{p}{p-1}} > \delta_1\lambda/4\}} |F|^{\frac{p}{p-1}(1+\sigma)} \,dx \Bigg) d\lambda \\
&=  \int_{B_{s_2r}\cap \{ |F|^{\frac{p}{p-1}}> \delta_1A\lambda_0 /4 \}}  |F|^{\frac{p}{p-1}(1+\sigma)} \left( \int_{A\lambda_0}^{4|F|^{\frac{p}{p-1}}/\delta_1}  \lambda^{\frac{q-p}{p}-1-\sigma}      d\lambda\right) dx\\
& =  \frac{1}{\frac{q-p}{p}-\sigma}  \int_{B_{s_2r}\cap \{ |F|^{\frac{p}{p-1}}> \delta_1A\lambda_0 /4 \}} |F|^{\frac{p}{p-1}(1+\sigma)}  \left[\Big(\frac{4|F|^{\frac{p}{p-1}}}{\delta_1}\Big)^{\frac{q-p}{p}-\sigma} -(A\lambda_0)^{\frac{q-p}{p}} \right] dx\\
&\le  \frac{c}{\delta_1^{\frac{q}{p}-1-\sigma}}\int_{B_{s_2r}} |F|^{\frac{q}{p-1}} dx. 
 \end{aligned}\end{equation}
  Inserting \eqref{est_duEkI_0}--\eqref{est_duEkI_2} into \eqref{est_duEkm}, we have that
   \begin{equation}\label{est_dudukEk}
\begin{aligned}
 \int_{E_k(KA\lambda_0, s_1r)} |du|^p |du|_k^{q-p}\;dx
 & \le  (KA\lambda_0)^{\frac{q-p}{p}} \int_{ B_{s_1r}}  |du|^p\; dx + c  \varepsilon\int_{B_{s_2r}} |du|^p |du|_k^{q-p} dx  \\
&\qquad  + c \delta_1^{-(\frac{q}{p}-1-\sigma)} \int_{B_{s_2r}}  |F|^{\frac{q}{p-1}}\; dx .
\end{aligned}\end{equation}

On the other hand, we see that 
 \begin{equation}\label{est_dudukcomp}
 \int_{B_{s_1r} \setminus E_k(KA\lambda_0, s_1r)} |du|^p |du|_k^{q-p}\;dx \leq (KA\lambda_0)^{\frac{q-p}{p}}    \int_{B_{s_1r}} |du|^p \;dx.
\end{equation}
Combining \eqref{est_dudukEk} and \eqref{est_dudukcomp}, it follows that 
$$\begin{aligned}
  & \int_{B_{s_1r}}  |du|^p |du|_k^{q-p}\;dx \\
  & \le  2(KA\lambda_0)^{\frac{q-p}{p}} \int_{ B_{s_1r}}  |du|^p\; dx 
+ c_*\varepsilon  \int_{B_{s_2r}} |du|^p |du|_k^{q-p} dx+ c \delta_1^{-(\frac{q}{p}-1-\sigma)} \int_{B_{s_2r}}  |F|^{\frac{q}{p-1}}\; dx 
\end{aligned}$$
for some $c_*=c_*(n,\ell,\nu, L, p,s)>0$.
At this stage, we ultimately select  $\varepsilon\in(0,1)$ such that $ c_*\varepsilon  \leq \frac12$, and hence $\delta_1$ is also finally determined as a constant depending on $n,\ell, \nu, L, p,s$. Recalling the definition of $A$ in \eqref{defA}, we conclude that
$$\begin{aligned}
 & \int_{B_{s_1r}}|du|_k^{q-p}|du|^{p}\, dx\\
 &\le \frac12 \int_{B_{s_2r}}|du|_k^{q-p} |du|^p\,dx +\frac{c\lambda_0^{\frac{q-p}{p}}}{(s_2-s_1)^{\frac{n(q-p)}{p}}}\int_{B_{2r}} |du|^{p}\, dx +c\int_{B_{2r}}|F|^{\frac{q}{p-1}}\,dx
\end{aligned}$$
for any $1\le s_1<s_2\le 2$, where the constants $c$ are independent of $s_1$ and $s_2$. Therefore Lemma \ref{teclem} below yields that
$$
 \int_{B_{r}}|du|_k^{q-p}|du|^{p}\, dx\leq c\lambda_0^{\frac{q-p}{p}}\int_{B_{2r}} |du|^{p}\, dx +c\int_{B_{2r}}|F|^{\frac{q}{p-1}}\,dx.
$$
Furthermore, we apply Lebesgue's monotone convergence theorem to get 
$$
 \fint_{B_{r}}|du|^q\, dx\leq c\lambda_0^{\frac{q-p}{p}}\fint_{B_{2r}} |du|^{p}\, dx +c\,\fint_{B_{2r}} |F|^{\frac{q}{p-1}}\,dx.
$$
Here, recalling the definition of $\lambda_0$ in \eqref{def_lambda0}, 
 we derive that  
\[ \begin{aligned}
 &\lambda_0^{\frac{q-p}{p}}\fint_{B_{2r}} |du|^{p}\, dx  \\
 &=  \left[ \fint_{B_{2r}} |du|^{p} \, dx + \frac{1}{
\delta_1}\left(\fint_{B_{2r}} |F|^{\frac{p}{p-1}(1+\sigma)} \, dx\right)^{\frac{1}{1+\sigma}} \right]^{\frac{q-p}{p}} \fint_{B_{2r}} |du|^{p}\, dx
\\
&
\le  \left( \fint_{B_{2r}} |du|^{p} \, dx \right)^{\frac{q}{p}}+ c \left(\fint_{B_{2r}} |F|^{\frac{p}{p-1}(1+\sigma)} \, dx \right)^{\frac{q-p}{p (1+\sigma)} \frac{q}{q-p}}  + c \left( \fint_{B_{2r}} |du|^{p}\, dx\right)^{\frac{q}{p}}\\
& \le c \left( \fint_{B_{2r}} |du|^{p} \, dx \right)^{\frac{q}{p}} +  c\fint_{B_{2r}} |F|^{\frac{q}{p-1}} \, dx,
\end{aligned}\]
by using Young's inequality and H\"older's inequality.
We finally obtain   
\[
 \fint_{B_{r}}|du|^q\, dx\leq c\left(\fint_{B_{2r}} |du|^{p}\, dx\right)^{\frac{q}{p}} +c\fint_{B_{2r}} |F|^{\frac{q}{p-1}}\,dx
\]
for some $c=c(n, \ell, \nu, L, p,s)>0.$ Therefore, we complete the proof of Theorem~\ref{mainthm}.

\begin{lem}\label{teclem} (see \cite{HL1})
Let $ g :[a,b] \to \mathbb{R} $ be a bounded nonnegative function. Suppose that for any $s_1,s_2$ with  $ 0< a \leq s_1 < s_2 \leq b $,
$$
g(s_1) \leq \tau g(s_2) + \frac{\alpha}{(s_2-s_1)^{\gamma}}+\beta,
$$
where $\alpha, \beta\geq 0, \gamma >0$ and $0\leq \tau <1$. Then we have
$$
g(s_1) \leq c\left(  \frac{\alpha}{(s_2-s_1)^{\beta}}+ \beta \right)
$$
for some constant $c=c(\gamma, \tau) >0$.
\end{lem}

Finally, we prove Corollary~\ref{cor1}.

\begin{proof}[Proof of Corollary~\ref{cor1}]
We first consider the case that $r=1$ and the center of balls is the origin. Then we observe from \cite[Theorem 13]{Sil17} that for every $\tilde f\in L^\gamma(B_2,\Lambda^\ell)$ with $d^*\tilde f=0$ in the distribution sense, there exists a unique weak solution $\tilde \theta\in W^{2,\gamma}(B_{2},\Lambda^\ell)$ to the equation
$$
\begin{cases}
d^* (d\tilde \theta) = \tilde f & \quad \text{in }\ B_{2},\\
\tilde \theta =0 & \quad \text{on }\ \partial B_{2},
\end{cases}
$$
with the estimate
$$
\|\theta\|_{W^{2,\gamma}(B_{2},\Lambda^\ell)} \le c \|\tilde f\|_{L^\gamma (B_{2},\Lambda^\ell)},
$$
where $c>0$ depends on $n,\ell$ and $\gamma$. 
Note that when applying   \cite[Theorem 13]{Sil17}, we consider that the ball $B_{2}$ is contractible, see \cite[Remark 17 (ii)]{Sil17} for details.  Set $\tilde F =d\tilde \theta$. Then by the Sobolev embedding, we have
$$
\left(\fint_{B_{2}}|\tilde F|^{\gamma^*}\, dx \right)^{\frac{1}{\gamma^*}}\le \left(\fint_{B_{2r}}|\nabla \tilde \theta |^{\gamma^*}\, dx \right)^{\frac{1}{\gamma^*}}  \le c \left(\fint_{B_{2}}|\tilde f|^{\gamma}\,dx\right)^{\frac{1}{\gamma}}
$$
for some $c>0$ depending on $n,\ell$ and $\gamma$, where $\gamma^*=\frac{n\gamma}{n-\gamma}$. 

Let us return to the setting of the corollary, fix $B_{2r}=B_{2r}(x_0)\Subset \Omega$ with $2r\le R$. Put $\tilde f(x)= r^2f(r(x-x_0))$,  $\tilde \theta(x)=\theta (r(x-x_0))$ and $F:= d \theta $. Then we have that $d^* F =f$ and  
$$
\left(\fint_{B_{2r}}| F|^{\gamma^*}\, dx \right)^{\frac{1}{\gamma^*}}  \le c \left(\fint_{B_{2r}}|r f|^{\gamma}\,dx\right)^{\frac{1}{\gamma}}.
$$
Therefore, by Theorem~\ref{mainthm} with $q=(p-1)\gamma^*$, we complete the proof.
\end{proof}

\section*{Acknowledgement}
M. Lee  was supported by NRF grant funded by MSIT (NRF-2022R1F1A1063032). J. Ok was supported by NRF grant funded by MSIT (NRF-2022R1C1C1004523).
J. Pyo was supported by NRF grant funded by MSIT (NRF-2020R1A2C1A01005698 and NRF-2021R1A4A1032418).

\subsection*{Data availability}
No data was used for the research described in the article.


\begin{thebibliography}{99}



\bibitem{AM05}
E.\ Acerbi and G. Mingione: 
 Gradient estimates for the  $p(x)$-Laplacean system,
J. Reine Angew. Math. 584 (2005), 117--148.

\bibitem{AM07}
E.\ Acerbi and G. Mingione: 
Gradient estimates for a class of parabolic systems,
Duke Math. J. 136 (2007), no. 2, 285--320.


\bibitem{BDGP22}
A.\ Balci, L. Diening, R. Giova and A. Passarelli di Napoli, Antonia:
Elliptic equations with degenerate weights.
SIAM J. Math. Anal. 54 (2022), no.2, 2373--2412.

\bibitem{BDS15}
S. Bandyopadhyay, B. Dacorogna and S. Sil:
 Calculus of variations with differential forms,
J. Eur. Math. Soc. (JEMS) 17 (2015), no. 4, 1009--1039.

\bibitem{BS16}
S. Bandyopadhyay and S. Sil:
Notions of affinity in calculus of variations with differential forms,
Adv. Calc. Var. 9 (2016), no. 3, 293--304.



\bibitem{BS13}
L.\ Beck and B.\ Stroffolini:
Regularity results for differential forms solving degenerate elliptic systems,
Calc. Var. Partial Differential Equations 46 (2013), no. 3-4, 769--808. 




\bibitem{BW12}
S. Byun and L. Wang:
Nonlinear gradient estimates for elliptic equations of general type, 
Calc. Var. Partial Differential Equations 45 (2012), no.3-4, 403--419.

\bibitem{CP98}
L. Caffarelli and I. Peral,
On $W^{1,p}$ estimates for elliptic equations in divergence form,
Comm. Pure Appl. Math. 51 (1998), no.1, 1--21.


\bibitem{CM16}
M. Colombo and G. Mingione,
Calderón-Zygmund estimates and non-uniformly elliptic operators,
J. Funct. Anal. 270 (2016), no. 4, 1416--1478.

\bibitem{CDK12} G.\ Csat\'o, B.\ Dacorogna, and O.\ Kneuss:
The pullback equation for differential forms.
Progr. Nonlinear Differential Equations Appl., 83
Birkh\"auser/Springer, New York, 2012. xii+436 pp.


\bibitem{CDS18}
G. Csato, B. Dacorogna and S. Sil:
On the best constant in Gaffney inequality,
J. Funct. Anal. 274 (2018), no. 2, 461--503.

\bibitem{DM93}
E.\ DiBenedetto and J.\ Manfredi:
On the higher integrability of the gradient of weak solutions of certain degenerate elliptic systems,
Amer. J. Math. 115 (1993), no.5, 1107--1134.




\bibitem{FG16}
M. Foss and C. S. Goodrich:
On partial Hölder continuity and a Caccioppoli inequality for minimizers of asymptotically convex functionals between Riemannian manifolds,
Ann. Mat. Pura Appl. (4) 195 (2016), no.5, 1405--1461.

\bibitem{GH04}
M. Giaquinta and M. Hong:
Partial regularity of minimizers of a functional involving forms and maps,
NoDEA Nonlinear Differential Equations Appl. 11 (2004), no. 4, 469--490.


\bibitem{Giusti} E.\ Giusti: 
\emph{Direct Methods in the Calculus of Variations}, 
World Scientific, Singapore, 2003.




\bibitem{Ham92}
C.\ Hamburger: 
Regularity of differential forms minimizing degenerate elliptic functionals, 
J. Reine Angew. Math. 431 (1992), 7--64.


\bibitem{Ham05}
C.\ Hamburger: 
The heat flow in nonlinear Hodge theory,
Adv. Math. 190 (2005), no. 2, 360--424.



\bibitem{HL1} 
Q.\ Han and F.\ Lin:
\emph{Elliptic Partial Differential Equations},  Courant Lecture Notes in Mathematics, 1, New York University, Courant Institute of Mathematical Sciences, New York, American Mathematical Society, Providence, RI, 1997.


\bibitem{Iwa83} 
T.\ Iwaniec:
Projections onto gradient fields and $L^p$-estimates for degenerated elliptic operators,
Studia Math. 75 (1983), no.3, 293--312.

\bibitem{IL93}
T.\ Iwaniec and A.\ Lutoborski:
Integral estimates for null Lagrangians, 
Arch. Rational Mech. Anal. 125 (1993), no. 1, 25--79. 

\bibitem{ISS99}
T.\ Iwaniec, C.\ Scott, C. and B.\ Stroffolini: 
Nonlinear Hodge theory on manifolds with boundary,
Ann. Mat. Pura Appl. (4) 177 (1999), 37--115. 


\bibitem{KZ99}
J.\ Kinnunen and S.\ Zhou:
A local estimate for nonlinear equations with discontinuous coefficients,
Comm. Partial Differential Equations 24 (1999), no.11-12, 2043--2068.


\bibitem{KMin14}
T.\ Kuusi and G.\ Mingione:
A nonlinear Stein theorem,
Calc. Var. Partial Differential Equations 51 (2014), no.1-2, 45--86.

\bibitem{KS23}
D.\ Kumar and S.\ Sil:
BMO estimates for Hodge-Maxwell systems with discontinuous anisotropic coefficients, arXiv:2310.06615




\bibitem{Sil17}
S.\ Sil:
Regularity for elliptic systems of differential forms and applications, 
Calc. Var. Partial Differential Equations 56 (2017), no. 6, Paper No. 172, 35 pp. 


\bibitem{Sil19}
S.\ Sil:
Nonlinear Stein theorem for differential forms,
Calc. Var. Partial Differential Equations 58 (2019), no. 4, Paper No. 154, 32 pp. 

\bibitem{Sil19-2}
S.\ Sil:
Calculus of variations: a differential form approach,
Adv. Calc. Var. 12 (2019), no. 1, 57--84.


\bibitem{Uhl77}
K.\ Uhlenbeck:
Regularity for a class of non-linear elliptic systems,
Acta Math.138(1977), no.3-4, 219–240.



%
%
%
%
%
%
%
%
%
%
%
%
%
%
%
%
%
%
%
%
%
%
%
%
%
%
%
%
%
%
%
%
%
%
%
%
%
%
%
%
%
%
%
%
%
%
%
%
%
%
%
%
%
%
%
%
%
%
%
%
%

%

%
%
%
%
%
%
%


%
%
%
%
%
%
%
%
%
%
%
%
%
%
%
%
%
%
%
%
%
%
%
%




%
%
%
%
%
%
%
%
%
%
%
%
%
%
%
%
%
%
%
%
%

\end{thebibliography}
\end{document}